\documentclass[12pt]{amsart}
\usepackage{amscd,amsmath,amssymb,amsfonts}
\usepackage[cmtip, all]{xy}

\newtheorem{thm}{Theorem}
\newtheorem{prop}[thm]{Proposition}
\newtheorem{lem}[thm]{Lemma}

\theoremstyle{definition}
\newtheorem{Def}{Definition}
\theoremstyle{remark}

\renewcommand{\theIntrorems}{{\kern-5pt}}
\numberwithin{equation}{section}

\newcommand{\sC}{{\mathcal C}}

\newcommand{\sM}{{\mathcal M}}

\newcommand{\sO}{{\mathcal O}}
\newcommand{\sP}{{\mathcal P}}
\newcommand{\sQ}{{\mathcal Q}}
\newcommand{\sR}{{\mathcal R}}

\newcommand{\C}{{\mathbb C}}

\renewcommand{\P}{{\mathbb P}}
\newcommand{\Q}{{\mathbb Q}}

\newcommand{\Z}{{\mathbb Z}}

\renewcommand{\L}{{\mathbb{L}}}

\newcommand{\Spec}{{\rm Spec \,}}

\newcommand{\Sch}{{\operatorname{\mathbf{Sch}}}}

\newcommand{\Sm}{{\mathbf{Sm}}}

\newcommand{\ds}{{/\kern-3pt/}}

\author{Y.-P. Lee and R. Pandharipande}

\date{February 2010}

\begin{document}


\title{Algebraic cobordism of bundles on varieties}

\begin{abstract}
The double point relation defines a natural
theory of algebraic cobordism for bundles on varieties.
We construct a simple basis (over $\Q$) of the
corresponding cobordism groups over $\text{Spec}(\C)$ for all dimensions of
varieties and ranks of bundles. The basis consists
of split bundles over products of projective spaces.
Moreover, we prove the full theory for bundles on varieties
is an extension of scalars of standard algebraic cobordism.
\end{abstract}

 \maketitle

\setcounter{tocdepth}{1}
\tableofcontents

\section*{Introduction}

\subsection{Algebraic cobordism}
A successful theory of algebraic cobordism 
has been constructed in \cite{ACFin}
from Quillen's axiomatic perspective. 
The result $\Omega_*$ is the
universal oriented Borel-Moore homology theory of schemes, 
yielding the universal oriented Borel-Moore cohomology 
theory $\Omega^*$ for the subcategory of smooth 
schemes.

Let $k$ be a field of characteristic 0. 
Let $\Sch_k$ be the category of 
 separated schemes of finite type over $k$, and
let $\Sm_k$ be the full subcategory of smooth quasi-projective $k$-schemes.
A geometric presentation of algebraic cobordism in characteristic
0 via
double point relations is given in \cite{LevP}.

\subsection{Double point degenerations}
Let $Y\in \Sm_k$ be of pure dimension. A morphism
\[
\pi:Y\to \P^1
\]
is a {\em double point degeneration}
over  $0\in \P^1$ if $\pi^{-1}(0)$ can be written as
\[
\pi^{-1}(0)=A \cup B 
\]
where 
$A$ and $B$ are smooth codimension one closed subschemes of $Y$,  intersecting transversely. 
The intersection 
\[
D=A\cap B
\]
is the {\em double point locus} of $\pi$ over $0\in \P^1$.
We do not require $A$, $B$, or $D$ to be connected. Moreover, 
 $A$, $B$, and $D$ are allowed to be empty.

Let $N_{A/D}$ and $N_{B/D}$ denote the normal bundles of $D$ in $A$ and $B$
respectively.
Since 
$O_D(A+B)$ is trivial,
$$N_{A/D} \otimes N_{B/D} \cong O_D.$$
Since
 $O_D\oplus N_{A/D}\cong N_{A/D}\otimes(O_D\oplus N_{B/D})$,  the projective bundles
\begin{equation}\label{pbd}
\P(O_D\oplus N_{A/D})\to D \ \ \text{and} \ \ \P(O_D\oplus N_{B/D})\to D
\end{equation}
are isomorphic. Let 
\[
\P(\pi) \rightarrow D
\]
denote either of \eqref{pbd}.

\subsection{$\sM(X)^+$}
For $X\in \Sch_k$,
let $\sM(X)$ denote the set of isomorphism classes over $X$ of
projective morphisms
\begin{equation} \label{fyx}
f:Y \rightarrow X
\end{equation}
with $Y\in \Sm_k$. The set $\sM(X)$ is a monoid under
disjoint union of domains and is graded by the dimension of $Y$ over $k$. Let $\sM_*(X)^+$ denote
the graded group completion of $\sM(X)$. 

Alternatively, $\sM_n(X)^+$ is the
free abelian group generated by morphisms \eqref{fyx}
where $Y$ is irreducible and of dimension $n$ over $k$. Let 
$$[f:Y \to X] \in \sM_*(X)^+$$
denote the element determined by the morphism.

\subsection{Double point relations}
Let $X\in \Sch_k$, and let 
$p_1$ and $p_2$ denote the projections to the first and second factors of 
 $X \times \P^1$ respectively.

Let $Y\in \Sm_k$ be of pure dimension. Let
\[
g: Y \to X \times \P^1
\]
be a projective morphism
for which the 
composition
\begin{equation}\label{pi2}
\pi = p_2\circ g: Y \to \P^1
\end{equation}
is 
a double point degeneration over $0 \in \P^1$. 
Let 
\[
[A\to X],\ [B\to X], \  [\P(\pi) \to X] \in \sM(X)^+
\]
be obtained from the fiber $\pi^{-1}(0)$ and the
morphism $p_1\circ g$.

\begin{Def} Let  $\zeta\in \P^1(k)$ be a regular value  of $\pi$. 
We call the map $g$ a {\em double point cobordism} with degenerate fiber over $0$ and smooth fiber over $\zeta$.
The associated
{\em double point relation} over $X$ is
\begin{equation}\label{dpr}
[Y_\zeta \to X] -[A \to X] - [B \to X] + [\P(\pi) \to X] 
\end{equation}
where $Y_\zeta = \pi^{-1}(\zeta)$. 
\end{Def} 
\vspace{8pt}

The relation \eqref{dpr} depends not only on the morphism 
$g$ and the point $\zeta$, but also on the choice of 
decomposition of the fiber 
$$\pi^{-1}(0)=A\cup B.$$
We view \eqref{dpr} as an analog of the classical relation 
of rational equivalence of 
algebraic cycles.

Let $\mathcal{R}_*(X)\subset \sM_*(X)^+$ be the subgroup generated
by {\em all} double point relations over $X$.
Since \eqref{dpr} is a homogeneous element of $\sM_*(X)^+$, 
 $\mathcal{R}_*(X)$ is a graded subgroup of $\sM_*(X)^+$.

\begin{Def}\label{Def:DPCobord} For $X\in\Sch_k$, {\em double point 
cobordism} $\omega_*(X)$ is defined by  the quotient
\begin{equation}\label{gttr}
\omega_*(X)=\sM_*(X)^+/\sR_*(X).
\end{equation}
\end{Def}
\vspace{8pt} 

A central result of \cite{LevP} is the isomorphism
\begin{equation}\label{ntt3}
\Omega_* \cong \omega_*\ 
\end{equation}
which provides a geometric presentation of algebraic
cobordism. Since resolution of singularities and Bertini's results
are used, the isomorphism is established only when $k$ has characteristic 0.

\subsection{Over a point} 
We write 
$\Omega_*(k)$ and $\omega_*(k)$ for 
$\Omega_*(\Spec(k))$ and $\omega_*(\Spec(k))$ respectively.
Let $\L_*$ be the Lazard ring \cite{Lazard}.
The canonical map
$$\L_* \to \Omega_*(k)$$
classifying the formal group law for $\Omega_*$ 
is proven to be an isomorphism in \cite[Theorem 4.3.7]{ACFin}.
By Quillen's result for complex cobordism (in topology),
 $$\L_n\cong MU^{-2n}(\text{pt}),$$ and the well-known generators of 
$MU^*(\text{pt})_\Q$ \cite[Chapter VII]{Stong}, we see 
 $\Omega_*(k)\otimes_{\mathbb{Z}}\Q$ is generated as a $\Q$-algebra by the classes of projective spaces.
The following result is then a consequence of \eqref{ntt3},
\begin{equation}\label{hhtt}
\omega_*(k)\otimes_{\mathbb Z} {\mathbb Q} 
=  \bigoplus_{\lambda} \  {\mathbb Q} [\P^{\lambda_1} \times ... 
\times \P^{\lambda_{\ell(\lambda)}}]\ ,
\end{equation}
where the sum is over all partitions $\lambda$.
The partition $\lambda=\emptyset$ corresponds to $[\P^0]$
in grade $0$.

\subsection{Bundles}
For $X\in \Sch_k$,
let $\sM_{n,r}(X)$ denote the set of isomorphism classes over $X$ of pairs
\begin{equation*} 
[f:Y \rightarrow X,  E]
\end{equation*}
with $Y\in \Sm_k$ of dimension $n$, $f$ projective, and $E$ a
rank $r$ vector bundle on $Y$.
The set $\sM_{n,r}(X)$ is a monoid under
disjoint union of domains. Let $\sM_{n,r}(X)^+$ denote
the group completion of $\sM_{n,r}(X)$.

Double point relations are easily defined in the setting
of pairs following \cite[Section 13]{LevP}. Let $Y\in \Sm_k$ be of pure dimension $n+1$. Let
\[
g: Y \to X \times \P^1
\]
be a projective morphism
for which the 
composition
\begin{equation*}
\pi = p_2\circ g: Y \to \P^1
\end{equation*}
is 
a double point degeneration over $0 \in \P^1$. 
Let $E$ be a rank $r$ vector bundle on $Y$.
Let 
\[
[A\to X, E_A],\ [B\to X, E_B], \  [\P(\pi) \to X, E_{\P(\pi)}] \in 
\sM_{n,r}(X)^+
\]
be obtained from the fiber $\pi^{-1}(0)$ and the
morphism $p_1\circ g$.
Here, $E_A$ and $E_B$ denote the restrictions of $E$ to $A$ and $B$
respectively.
The restriction $E_{\P(\pi)}$ is defined by pull-back from $Y$
via
$$\P(\pi) \rightarrow D \subset Y\ .$$

\begin{Def} Let  $\zeta\in \P^1(k)$ be a regular value  of $\pi$. 
The associated
{\em double point relation} over $X$ is
\begin{equation*}
[Y_\zeta \to X, E_{Y_\zeta}] -[A \to X, E_A] - [B \to X, E_B] + 
[\P(\pi) \to X, E_{\P(\pi)}] 
\end{equation*}
where $Y_\zeta = \pi^{-1}(\zeta)$. 
\end{Def} 
\vspace{8pt}

For $X\in \Sch_k$,
let $\mathcal{R}_{n,r}(X) \subset M_{n,r}(X)^+$
be the subgroup generated by all double point relations. 
Double point cobordism theory for bundles on varieties
is defined by
$${\omega}_{n,r}(X) = 
\sM_{n,r}(X)^+/\mathcal{R}_{n,r}(X).$$
The sum
$$\omega_{*,r}(X) = \bigoplus_{n=0}^\infty \omega_{n,r}(X)$$
is always a $\omega_*(k)$-module via product (and pull-back).
If $X\in \Sm_k$, then  $\omega_{*,r}(X)$ is also a
module over the ring $\omega_{*}(X)$.

\subsection{Basis} \label{bbb3}
The main result of the paper is the construction of a basis of
$\omega_{n,r}(k)$ analogous to the fundamental presentation
\eqref{hhtt}.
Our  basis is indexed by pairs of  partitions. 
A {\em partition pair} of size $n$ and type $r$ is a pair $(\lambda,\mu)$
where
\begin{enumerate}
\item[(i)] $\lambda$ is a partition of $n$,
\item[(ii)] $\mu$ is a sub-partition of $\lambda$ of length $\ell(\mu)\leq r$.
\end{enumerate}
The sub-partition condition means $\mu$ is obtained by
deleting parts of $\lambda$. The partition $\mu$ may be empty and
may equal $\lambda$ if $\ell(\lambda)\leq r$.
Sub-partitions $\mu,\mu' \subset \lambda$ are
equivalent if they differ by permuting equal parts of $\lambda$.

Let $\mathcal{P}_{n,r}$ be the set of all partition pairs of
size $n$ and type $r$.
For example,
$$
\mathcal{P}_{3,2} = \left\{ \begin{array}{c} 
(3,\emptyset), \ (3,3), \\ (21, \emptyset), \
(21,2), \ (21,1), \ (21,21),\ \\
 (111, \emptyset), \ (111,1),\ (111,11) \end{array} \right\}\ .
$$

To each $(\lambda,\mu) \in \mathcal{P}_{n,r}$, we associate
an element 
$$\phi(\lambda,\mu) \in \omega_{n,r}(k)$$ by the following construction.
Let
$\P^\lambda = \P^{\lambda_1} \times \ldots \times 
\P^{\lambda_{\ell(\lambda)}}$.
To each part $m$ of $\mu$, let
$$L_{m} \rightarrow \P^\lambda$$
be the line bundle obtained by pulling-back  $\mathcal{O}_{\P^{m}}(1)$ via the
projection to the factor
$$\P^\lambda \rightarrow \P^m\ $$
corresponding to the part $m$. Since $\mu \subset \lambda$, $m$ is part of $\lambda$.
We define
$$\phi(\lambda,\mu)= [\P^\lambda, \ \mathcal{O}^{r-\ell(\mu)} \oplus
\bigoplus
_{m\in \mu} L_m\ ]\ .$$
The bundle on $\P^\lambda$ has a trivial factor of rank $r-\ell(\mu)$.

\begin{thm}\label{one}
For $n,r \geq 0$, we have
$$\omega_{n,r}(k) \otimes_{\mathbb{Z}}\Q = \bigoplus_{(\lambda,\mu) \in
\mathcal{P}_{n,r}} \Q \cdot \phi(\lambda,\mu)\ .$$
\end{thm}

In other words, the elements $\phi(\lambda,\mu)$ determine
a basis of $\omega_{n,r}(k)\otimes_{\mathbb{Z}}\Q$.
If $r=0$, Theorem \ref{one} specializes to  \eqref{hhtt}.
In case $(n,r)=(3,2)$, the basis of Theorem 1 is given by

\vspace{8pt}
\begin{math} \begin{array}{rl}
& {}[ \P^3, \sO^2],\\  
& {}[\P^3, \sO\oplus \sO(1)], \\
& {}[\P^{2}\times \P^1, \sO^2], \\
& {}[\P^{2}\times \P^1, \sO \oplus \sO(1,0)],  \\
& {}[\P^{2}\times \P^1, \sO \oplus \sO(0,1)], \\ 
& {}[\P^{2}\times \P^1, \sO(1,0) \oplus \sO(0,1)], \\
& {}[\P^1\times \P^1 \times \P^1, \sO^2],  \\
& {}[\P^1\times \P^1 \times \P^1, \sO\oplus \sO(1,0,0)], \\
& {}[\P^1\times \P^1 \times \P^1, \sO(1,0,0) \oplus\sO(0,1,0)] \ . 
\end{array}\end{math}
\vspace{8pt}

Theorem \ref{one} is proven in Section \ref{hr}. The argument 
requires studying an algebraic cobordism theory for lists
of line bundles on varieties developed in  Section \ref{llll}.

The structure of $\omega_{*,r}(k)$ over $\Z$ is determined
by the following result proven in Section \ref{lltt2}.

\begin{thm}\label{lltt}
For $r \geq 0$, $\omega_{*,r}(k)$ is a free $\omega_{*}(k)$-module
with basis 
$$\omega_{*,r}(k) = \bigoplus_{\lambda}\  \omega_{*}(k) \cdot \phi(\lambda,\lambda)$$
where the sum is over all partitions $\lambda$ of length at
most $r$.
\end{thm}

\subsection{Over $X$}
In fact, $\omega_{*,r}(k)$ determines $\omega_{*,r}(X)$ for
all $X\in \Sch_k$. There is a natural map
$$\gamma_X: \omega_*(X) \otimes_{\omega_*(k)} \omega_{*,r}(k) \rightarrow \omega_{*,r}(X) $$
of $\omega_*(k)$-modules defined by
\begin{multline*}
\gamma_X\Big( [Y \stackrel{f}{\rightarrow} X] \otimes \phi(\lambda,\lambda) \Big) =
\\  [Y\times \P^\lambda\stackrel{f\circ p_Y}\longrightarrow X, 
\ \mathcal{O}^{r-\ell(\lambda)} \oplus
\bigoplus
_{m\in \lambda} p_{\P^\lambda}^*(L_m)\ ]\ .
\end{multline*}
Here, $\lambda$ has length at most $r$, and 
$p_Y$ and $p_{\P^\lambda}$ are the projections of
 $Y \times \P^\lambda$ to $Y$ and $\P^\lambda$ respectively.

\begin{thm}\label{uuu} For $r\geq 0$ and  $X\in \Sch_k$, the map $\gamma_X$
is an isomorphism of $\omega_*(k)$-modules.
\end{thm}

By Theorem \ref{uuu}, the algebraic cobordism theory
$\omega_{*,r}$ of bundles on varieties is simply an
extension of scalars of the original theory $\omega_*$.

\subsection{Chern invariants}
let $Y$ be a nonsingular projective variety of dimension
$n$, and let $E$ be a rank $r$ vector bundle on $Y$. The
Chern invariants of the pair $[Y,E]$ are
$$\int_Y \Theta\Big(c_1(T_Y), \ldots, c_n(T_Y), c_1(E), \ldots, c_r(E)\Big)$$
where $\Theta$ is any graded degree $n$ 
polynomial  (with $\Q$-coefficients)
of the Chern classes of the tangent bundle $T_Y$ and $E$.

Let $\mathcal{C}_{n,r}$ be the finite dimensional $\Q$-vector 
space of graded degree $n$
polynomials in the Chern classes. 

\begin{thm} \label{tttt} The Chern invariants respect algebraic cobordism. The
resulting map
$$\omega_{n,r}(k)\otimes_{\mathbb{Z}} \Q
 \rightarrow \mathcal{C}_{n,r}^*$$
is an isomorphism.
\end{thm}

A simple counting argument (given in 
Section \ref{mmmm1}) shows the dimension of
$\mathcal{C}_{n,r}$ equals the cardinality
of $\mathcal{P}_{n,r}$. In case $(n,r)=(3,2)$, there are
9 basic Chern invariants of $[Y,E]$,
$$c_3(T_Y), \ \ c_{2}(T_Y)c_1(T_Y),\ \ c_1(T_Y)^3, \ \
c_2(T_Y)c_1(E), \ \ c_1(T_Y)^2c_1(E),$$
$$ c_1(T_Y),c_2(E),\ \ c_1(T_Y)c_1(E)^2,\ \ c_2(E)c_1(E),\ \ c_1(E)^3\ .$$
Theorem \ref{tttt} is proven jointly with Theorem \ref{one} in Section \ref{hr}.

\subsection{Applications}
For studying a theory associated to pairs $[Y,E]$ which admits
a multiplicative double point degeneration formula, algebraic cobordism
$\omega_{n,r}(\C)$ 
is a useful tool. The full theory can be calculated from
the toric basis elements specified by Theorem \ref{one}.

The determinations of $\omega_3(\C)$ and $\omega_{2,1}(\C)$ have been used in
\cite{LevP} to prove the conjectures of \cite{BP,MNOP1,MNOP2} 
governing the degrees of virtual
classes on the Hilbert schemes of points of 3-folds. 
Recently, Y. Tzeng \cite{tz} has used the 4-dimensional basis of 
$\omega_{2,1}(\C)$ in
a beautiful proof of
G\"ottsche's conjecture \cite{lothar} governing nodal curve counting
(interpreted as degrees of cycles in the
Hilbert schemes of points of surfaces).
The basis of $\omega_{n,r}(\C)$ 
will be  used in \cite{ypl} for the study of flop
invariance of quantum cohomology.

\subsection{Speculations}
Consider the algebraic group  $\mathbf{GL}_r$  over $\C$.
We view $\omega_{*,r}(\C)$ as an algebraic model for
$MU_*(B\mathbf{GL}_r)$. Theorem \ref{tttt} may be interpreted
as saying 
$\omega_{*,r}(\C)$ is dual to 
$$MU^*(B\mathbf{GL}_r) = MU^*(\text{pt})[[c_1,\ldots, c_r]]\ .$$
D. Maulik suggests defining an algebraic cobordism
theory $\omega_{*,\mathbf{G}}$
for principal $\mathbf{G}$-bundles on algebraic varieties
by the double point relation of Definition 3. Perhaps the resulting theory
over a point for classical groups $\mathbf{G}$
is dual to $MU^*(B\mathbf{G})$?

An algebraic approach to $MU^*(B\mathbf{G})$ for linear algebraic
groups 
has been proposed in \cite{D} by limits of
$\Omega^*$ over algebraic approximations to $B\mathbf{G}$. The
construction is similar to Totaro's definition \cite{T} of the
Chow ring of $B\mathbf{G}$, but requires also the coniveau
filtration (see \cite{ml} for an alternative
limit definition).  For many examples, including $B\mathbf{GL}_r$,
the isomorphism
$$\Omega^*(B\mathbf{G}) \cong MU^*(B\mathbf{G})$$
is obtained \cite{D}. Such isomorphisms were predicted in \cite{y}.

Another approach to our paper is perhaps possible
via a limit definition of $\Omega_{*}(B\mathbf{GL}_r)$.
There should be a map
$$\Omega_{*}(B\mathbf{GL}_r) \rightarrow \omega_{*,r}(\C)$$
which is injective by Chern invariants and surjective
by Proposition \ref{popop}.

\subsection{Acknowledgments}

We thank D. Abramovich, 
 J. Li, D. Maulik, B. Totaro,
and Y. Tzeng for discussions about
algebraic cobordism and double point degenerations. 
The basis of Theorem \ref{one} was guessed while
writing \cite{LevP}. 
Conversations with M. Levine played an essential role.
He suggested the possibility
of the extension of scalars result established in Theorem \ref{uuu}.

Y.-P. L. was 
supported by NSF grant  DMS-0901098. 
R. P. was supported by NSF
 grant DMS-0500187. 

\section{Chern classes}
\label{mmmm1}
\subsection{Cobordism invariance}
Let $n,r\geq 0$.
There is canonical bilinear map
$$\rho: \sM_{n,r}(k)^+ \otimes_{\Z}\Q \ \times\ \sC_{n,r} \rightarrow \Q$$
defined by integration,
$$\rho([Y,E], \Theta) = \int_Y \Theta\Big(c_1(T_Y), \ldots, c_n(T_Y), c_1(E), \ldots, c_r(E)\Big)\ .$$

\begin{prop} \label{yoyo}
The pairing $\rho$ annihilates $\mathcal{R}_{n,r}(k)$.
\end{prop}

\begin{proof} In case $r=0$, the invariance of 
the Chern numbers of the tangent bundle is a well-known property
of algebraic cobordism over a $\Spec(\C)$, see \cite{ACFin,Stong}.

 Let $Y\in \Sm_k$ be of pure dimension $n+1$. Let
\begin{equation} \label{gbt}
\pi: Y \to \P^1
\end{equation}
be a projective morphism which is
a double point degeneration over $0 \in \P^1$. 
Let $L$ be a line bundle on $Y$.
Suppose $L$ is very ample on $Y$.
Cutting $Y$ with  $s$ generic sections of $L$ yields
an nonsingular subvariety of codimension $s$,
$$ D^s\subset Y \stackrel{\pi}{\rightarrow} \P^1 \ .$$
The composition $D^s \rightarrow \P^1$ is 
a double point degeneration over $0\in \P^1$.

Let $Y_\zeta$, $A$, $B$, and $\P(\pi)$ be the four spaces
which occur in the double point relation 
for \eqref{gbt}
in Definition 3. Let 
$$Y_\zeta \cap D^s,\ \  A\cap D^s,\ \ B\cap D^s, \ \  \P(\pi)\cap D^s$$
be the four spaces which occur in the
relation for $D^s\rightarrow \P^1$.
Since the tangent bundle of $Z\cap D^s$ 
satisfies
$$0 \rightarrow T_{Z\cap D^s} \rightarrow T_{Z}|_{Z\cap D^s} \rightarrow
\bigoplus_{i=1}^s L \rightarrow 0$$
in each of the four cases,
we have 
$$c(T_{Z\cap D^s}) = \frac{c(T_Z)}{(1+c_1(L))^s}\ ,$$
$$  
c_i(T_{Z\cap D^s}) = c_i(T_Z)- s\cdot c_{i-1}(T_Z) c_1(L) + \cdots \ , 
$$
where we have suppressed the restrictions.
The application of the $r=0$ case of the Proposition to
the degenerations
$D^s\rightarrow \P^1$ for all $s$ implies (by descending
induction) the $r=1$ case for 
double point relations where $L$ is ample.

Similarly if $L_1, \ldots, L_m$ are very ample line
bundles on $Y$, we can consider 
$$ D^{s_1,\ldots, s_m}\subset Y \stackrel{\pi}{\rightarrow} \P^1 \ $$
obtained by cutting with $s_1$ sections of $L_1$, 
$s_2$ sections of $L_2$, \ldots, and $s_m$ sections of $L_m$.
The application of the $r=0$ case of the Proposition to
the degeneration
$D^{s_1,\ldots, s_2}\rightarrow \P^1$ for all $s_1, \ldots,s_m$ 
implies invariance under the double point relation
of graded degree $n$ polynomials in
the Chern classes of the tangent bundle and the
Chern classes of $L_1, \ldots, L_m$. 

The $r=1$
case of the Proposition follows since every line bundle $L$
may be written as the difference of two very ample line
bundles.

To prove the $r>1$ case of the Proposition, we use a
splitting argument. Let $\pi$ be a double point
degeneration as above \eqref{gbt}.
Let $E$ be a rank $r$ bundle on $Y$. Let
$${\mathbb{F}}(E) \rightarrow Y \stackrel{\pi}\rightarrow \P^1$$
by the complete flag variety over $Y$ obtained from $E$.
The composition ${\mathbb{F}}(E) \rightarrow \P^1$
is a double point degeneration with tautological line bundles
$L_1, \ldots, L_{r}$ which sum in $K$-theory to the
pull-back of $E$. The established line bundle
results then yield the $r>1$ case. 
\end{proof}

As a consequence of Proposition \ref{yoyo}, the pairing
$\rho$ descends,
\begin{equation} \label{kyyx}
\rho: \omega_{n,r}(k)\otimes_{\Z}\Q \ \times \ \sC_{n,r} \rightarrow
\Q \ . 
\end{equation}
Our first goal is to bound the rank of the pairing from below.

\subsection{Independence}

\subsubsection{Monomials of $\sC_{n,r}$} \label{mmmm}

For notational convenience, we write elements $\Theta\in \sC_{n,r}$
as polynomials
$$\Theta(u_1, \ldots, u_n, v_1, \ldots, v_r)$$
where $u_i= c_i(T_Y)$ and $v_i= c_i(E)$. Both $u_i$ and $v_i$
have degree $i$.
A canonical basis of $\sC_{n,r}$ is obtained
by monomials of graded degree $n$.

Let $\sQ_{n,r}$ be the set of partition pairs $(\nu,\mu)$
where 
\begin{enumerate}
\item[(i)] $\mu$ is a partition of size $|\mu| \leq n$ with largest part
at most $r$,
\item[(ii)]$\nu$ is a partition of $n-|\mu|$.
\end{enumerate}
The correspondence
\begin{equation} \label{nntt34}
\prod_{i=1}^n u_i^{l_i} \prod_{j=1}^r v_j^{m_j} \ \ \leftrightarrow
\ \ (1^{l_1} \cdots n^{l_n}, \ 1^{m_1} \cdots r^{m_r})
\end{equation}
yields a bijection between the monomial basis of $\sC_{n,r}$
and the set $\sQ_{n,r}$. Let
${\mathsf{C}}(\nu,\mu)$ 
denote the monomial associated to $(\nu,\mu) \in \sQ_{n,r}$.

\begin{lem} There is a natural bijection 
$\epsilon: \sQ_{n,r} \rightarrow \sP_{n,r}$ .
\end{lem}

\begin{proof}
Given $(\nu,\mu) \in \sQ_{n,r}$, define 
$$\epsilon(\nu,\mu) = (\nu \cup \mu^t, \mu^t) \in \sP_{n,r} \ .$$
Here, $\mu^t$ is the partition obtained by transposing the
Young diagram associated to $\mu$. Hence, $\mu^t$ has length
at most $r$. 
\end{proof}

\subsubsection{Ordering}

The $v$-degree of a monomial in $\sC_{n,r}$
is the vector
$$\deg_v
\left( \prod_{i=1}^n u_i^{l_i} \prod_{j=1}^r v_j^{m_j} \right) 
 = (m_1, m_2, \ldots, m_r) \in \mathbb{Z}_{\ge 0}^r \ . $$
We 
define a total ordering on $\Z_{\geq 0}^r$
by the following rule:
$$(m_1,\ldots, m_r) > (m'_1, \ldots, m_r')$$
 if either $m_r > m'_r$ or
if $m_j=m'_j$ for all $j >i$ and $m_i>m'_i$.
The resulting partial order on the monomials on $\sC_{n,r}$ 
(indexed by $\sQ_{n,r}$)
is sensitive only the variables $v_i$.

\label{ggn}

\subsubsection{Bilinear pairing} \label{bilp}

Let ${\bf M}$ be the matrix with rows and columns
indexed by $\sQ_{n,r}$ and elements
$${\bf M}_{n,r}[(\nu,\mu), (\nu',\mu')] = \rho\Big( \phi\big(\epsilon(\nu,\mu)
\big), \ {\mathsf{C}}(\nu',\mu') \Big) \ $$
for $(\nu,\mu),(\nu',\mu') \in \sQ_{n,r}$. 
Recall, the map
$$\phi: \sP_{n,r} \rightarrow \omega_{n,r}(k)$$
was defined in Section \ref{bbb3}.
The rows and columns of ${\bf M}_{n,r}$ are ordered 
by the partial ordering on $\sQ_{n,r}$ defined
in Section \ref{ggn}.

\begin{lem} If $(\nu,\mu)<(\nu',\mu')$ in the partial
order of $\sQ_{n,r}$, then
$${\bf M}_{n,r}[(\nu,\mu), (\nu',\mu')] = 0 \ .$$ \label{jjyy}
\end{lem}

\begin{proof}
Let $\mu=1^{m_1}\cdots r^{m_r}$ and $\mu' =1^{m'_1} \cdots r^{m'_r}$.
If $(\nu,\mu)<(\nu',\mu')$, then, in the highest index $i$ where
a difference occurs, $m_i < m'_i$.

Suppose the difference occurs in the index $i=r$.
Then, $m_r$ is the minimal part of $\mu^t$.
For the pair $[Y,E]= \phi(\nu,\mu)$, the bundle
$E$ is a direct sum of $r$ line bundles
pulled-back from the $\sO(1)$ 
factors of a product of $r$ 
projective spaces (with minimal dimension $m_r$).
Since $m_r < m'_r$, the  class $c_r^{m'_r}(E)$
vanishes on $Y$ by dimension considerations.

If the highest difference occurs in an index $i<r$,
the argument is the same (following again
from elementary dimension considerations).
\end{proof}

\begin{prop} ${\bf M}_{n,r}$ is a nonsingular
matrix. \label{hyy2}
\end{prop}

\begin{proof}
By Lemma \ref{jjyy}, the matrix ${\bf M}_{n,r}$ is
block lower triangular with respect to the
partial ordering on $\sQ_{n,r}$. The blocks
are determined by all $(\nu,\mu) \in \sQ_{n,r}$
with the same $\mu$.

Let $\mu=1^{m_1} \ldots r^{m_r}$. Consider the bundle 
$$E=\bigoplus_{m\in \mu^t} L_m \longrightarrow \P^{\mu^t}\ ,$$
following the notation of Section \ref{bbb3}.
Since
\begin{equation} \label{ppp3}
\int_{\P^{\mu^t}} c_1(E)^{m_1}c_2(E)^{m_2} \ldots c_r(E)^{m_r} = 1\ ,
\end{equation}
the block in ${\bf M}_{n,r}$ corresponding to $\mu$
is the
matrix ${\bf M}_{n-|\mu|,0}$.
The latter is
nonsingular by well-known results about the usual
$r=0$ theory of algebraic cobordism \cite{ACFin,Stong}.
\end{proof}

As a consequence of Proposition \ref{hyy2}, 
the generators proposed in Theorem \ref{one}
 span a subspace of $\omega_{n,r}(k)\otimes_\Z \Q$ of rank at 
least $|\sP_{n,r}|$.
In particular,
$$\text{dim}\left( \omega_{n,r}(k)\otimes_\Z \Q
 \right)  \geq | \sP_{n,r}|\ .$$
Moreover, the pairing \eqref{kyyx} has rank at least $|\sP_{n,r}|$.
To complete the proofs of Theorem \ref{one} and \ref{tttt}, 
we will prove the reverse inequality 
$$\text{dim}\left( \omega_{n,r}(k)\otimes_\Z \Q
 \right)  \leq | \sP_{n,r}|\ $$
in Section \ref{hr}.

\section{Lists of line bundles}
\label{llll}
\subsection{Lists}

For $X\in \Sch_k$,
let $\sM_{n,1^r}(X)$ denote the set of isomorphism classes over $X$ of tuples
\begin{equation*} 
[f:Y \rightarrow X, L_1, \ldots, L_r]
\end{equation*}
with $Y\in \Sm_k$ of dimension $n$, $f$ projective, and $L_1, \ldots, L_r$ an
ordered list of 
line bundles on $Y$.
The set $\sM_{n,1^r}(X)$ is a monoid under
disjoint union of domains. Let $\sM_{n,1^r}(X)^+$ denote
the group completion of $\sM_{n,1^r}(X)$. 

Let $Y\in \Sm_k$ be of pure dimension $n+1$, and  let
\[
g: Y \to X \times \P^1
\]
be a projective morphism
for which the 
composition
\begin{equation*}
\pi = p_2\circ g: Y \to \P^1
\end{equation*}
is 
a double point degeneration over $0 \in \P^1$. 
Let $L_1, \ldots, L_r$ be a list of  line  bundles on $Y$.
Let 
\[
[A\to X, L_{1,A}, \ldots L_{r,A}],\ \ \   [B\to X, L_{1,B} \ldots
L_{r,B}], \]
$$\ \ \ \  \  [\P(\pi) \to X, L_{1,\P(\pi)}, \ldots,L_{1,\P(\pi)} ] \  \in 
\sM_{n,1^r}(X)^+
$$
be obtained from the fiber $\pi^{-1}(0)$ and the
morphism $p_1\circ g$.

\begin{Def} Let  $\zeta\in \P^1(k)$ be a regular value  of $\pi$. 
The associated
{\em double point relation} over $X$ is
\begin{multline*}
[Y_\zeta \to X, \{ L_{i,Y_\zeta}\} ] -[A \to X, \{L_{i,A}\}]\\ 
 - 
[B \to X, \{ L_{i,B}\}] + 
[\P(\pi) \to X,\{L_{i,\P(\pi)}\}] 
\end{multline*}
where $Y_\zeta = \pi^{-1}(\zeta)$. 
\end{Def} 
\vspace{8pt}

For $X\in \Sch_k$,
let $\mathcal{R}_{n,1^r}(X) \subset M_{n,1^r}(X)^+$
be the subgroup generated by all double point relations. 
Double point cobordism theory for lists of line bundles on varieties
is defined by
$${\omega}_{n,1^r}(X) = 
\sM_{n,1^r}(X)^+/\mathcal{R}_{n,1^r}(X).$$
The sum
$$\omega_{*,1^r}(X) = \bigoplus_{n=0}^\infty \omega_{n,1^r}(X)$$
is always a $\omega_*(k)$-module via product.
If $X\in \Sm_k$, then  $\omega_{*,1^r}(X)$ is also a
module over the ring $\omega_{*}(X)$.

\subsection{Basis} \label{bbb4}
A {\em partition list} of size $n$ and type $r$ is a tuple $(\lambda,(m_1,\ldots,m_r))$
where
\begin{enumerate}
\item[(i)] $\lambda$ is a partition of $n$,
\item[(ii)] $(m_1,\ldots,m_r)$ is a list with $m_i \geq 0$ whose
union of {\em non-zero} parts is a sub-partition $\mu \subset \lambda$.
\end{enumerate}

Let ${\mathcal{P}}_{n,1^r}$ be the set of all partition lists of
size $n$ and type $r$.
For example,
$$
{\mathcal{P}}_{3,1^2} = \left\{\begin{array}{c} 
 (3,(0,0)), \ (3,(3,0)),  \ (3,(0,3)), \\
(21, (0,0)), \
(21,(2,0)),  (21,(1,0)),  \\  (21,(0,1)), \ (21,(0,2)),
\ (21,(2,1)), \ (21,(1,2)), \\
 (111, (0,0)), \ (111,(1,0)),\ (111,(0,1)), \ (111,(1,1)) 
\end{array} \right\}. $$

To each $(\lambda,(m_1,\ldots,m_r)) \in \mathcal{P}_{n,1^r}$, we associate
an element 
$$\phi(\lambda,(m_1,\ldots,m_r)) \in \omega_{n,r}(k)$$ by the following construction.
Let
$\P^\lambda = \P^{\lambda_1} \times \ldots \times\P^{\lambda_{\ell(\lambda)}}$.
To each non-zero part $m_i$ , let
$$L_{m_i} \rightarrow \P^\lambda$$
be the line bundle obtained by pulling-back  $\mathcal{O}_{\P^{m_i}}(1)$ via the
projection to the factor
$$\P^\lambda \rightarrow \P^{m_i}\ $$
corresponding to the part $m_i$. 
If $m_i=0$, let $L_{m_i}$ be the trivial line bundle on $\P^\lambda$.
We define
$$\phi(\lambda,(m_1, \ldots, m_r))= [\P^\lambda, (L_{m_1}, \ldots,
L_{m_r})\ ]\ .$$

\begin{thm} \label{kk34}
For $n,r \geq 0$, we have
$$\omega_{n,1^r} \otimes_{\mathbb{Z}}\Q = \bigoplus_{(\lambda,(m_1,\ldots, m_r)) \in
{\mathcal{P}}_{n,1^r}} \Q \cdot \phi(\lambda,(m_1,\ldots,m_r))\ .$$
\end{thm}

Theorem \ref{kk34} will be proven in Section \ref{gaga} with a mix of
techniques from \cite{ACFin, LevP} and new methods for studying
algebraic cobordism relations for line bundles on varieties.

\subsection{Chern invariants} \label{ci}
Let ${\sC}_{n,1^r}$ be the $\Q$-vector space of graded degree $n$ 
polynomials in the Chern classes 
$$c_1(T_Y), \ldots, c_n(T_Y), c_1(L_1), \ldots, c_1(L_r)\ .$$ 
There is canonical bilinear map
$$\rho: \sM_{n,1^r}(k)^+ \otimes_{\Z}\Q\ \times\ {\sC}_{n,1^r} \rightarrow \Q$$
defined by integration,
$$\rho([Y,E], \Theta) = \int_Y \Theta\Big(c_1(T_Y), \ldots, c_n(T_Y), c_1(L_1), 
\ldots, c_1(L_r)\Big)\ .$$
The proof of Proposition \ref{yoyo} implies 
the pairing $\rho$ annihilates $\mathcal{R}_{n,1^r}(k)$.
Hence, 
$\rho$ descends,
\begin{equation*} 
\rho: \omega_{n,1^r}(k)\otimes_{\Z}\Q\ \times\ \sC_{n,1^r} \rightarrow
\Q \ . 
\end{equation*}

The monomial basis of $\sC_{n,1^r}$ is easily seen to have the
same cardinality as the set $\sP_{n,1^r}$. A straightforward
extension of the methods of Section \ref{bilp} implies the elements of
$$\{ \phi(\lambda,(m_1,\ldots,m_r)) \ | \ (\lambda,(m_1,\ldots, m_r)) \in 
\sP_{n,1^r} \} \ \subset \omega_{n,1^r}\otimes_\Z \Q$$
span a subspace of dimension $|\sP_{n, 1^r}|$. In particular,
$$\text{dim}( \omega_{n,1^r}\otimes_\Z \Q) \geq |\sP_{n,1^r}|\ .$$

\subsection{Globally generated line bundles}\label{glob}
Let $\mathfrak{m} \subset \omega_*(k)$ be the ideal generated
by all elements of positive dimension,
$$ 0 \rightarrow \mathfrak{m} \rightarrow \omega_*(k) \rightarrow \Z \rightarrow 0\ .$$
 Since $\omega_{*,1^r}(k)$ is a $\omega_*(k)$-module, we can define the
graded quotient 
$$\widetilde{\omega}_{*,1^r}(k)
 = \frac{\omega_{*,1^r}(k)}{\mathfrak{m}\cdot \omega_{*,1^r}(k)}
\ , 
\ \ \ \ \ \ 
\widetilde{\omega}_{*,1^r}(k)
 = \bigoplus_{n=0}^\infty  \widetilde{\omega}_{n,1^r}(k)\ .$$
For $(\lambda,(m_1,\ldots,m_r))\in \sP_{n,1^r}$, let
$$\ \widetilde{\phi}(\lambda,(m_1,\ldots,m_r)) \in 
\widetilde{\omega}_{n,1^r}(k)$$
denote the class of ${\phi}(\lambda,(m_1,\ldots,m_r))$ in the quotient.

\begin{prop} Let $Y\in \Sm_k$ be a projective variety 
of dimension $n$ with \label{popo}
line bundles $L_1, \ldots, L_r$ all generated by global sections.
Then,
$$[Y, L_1, \ldots, L_r] \in \widetilde{\omega}_{n,1^r}(k) $$
lies in the $\Z$-linear span of 
$$\left\{\ \widetilde{\phi}(\lambda,(m_1,\ldots,m_r))
 \ \Big|  \ (\lambda,(m_1,\ldots,m_r))\in
\sP_{n,1^r}, \
\sum_{i=1}^r m_i =n \ \right\}$$
in $\widetilde{\omega}_{n,1^r}(k)$.
\end{prop}

\begin{proof}
Since $L_1,\ldots, L_r$ are all generated by global sections on $Y$,
there exists a projective morphism
$$f: Y \rightarrow \P^{d_1} \times \cdots \times \P^{d_r}, \ \ \ L_i = f^*(\sO_{\P^{d_i}}(1))
\ . $$
We view $f$ as determining an element of algebraic cobordism,
$$[f: Y \rightarrow \P^{d_1} \times \cdots \times \P^{d_r}] \in
\omega_n(\P^{d_1} \times \cdots \times \P^{d_r})\ . $$
A fundamental result of \cite[Theorem 1.2.19]{ACFin}
is the isomorphism
\begin{equation}\label{gt5}
 A_*(X) \cong \widetilde{\omega}_{*}(X) =
\frac{\omega_{*}(X)}{\mathfrak{m}\cdot \omega_{*}(X)}\ ,
\end{equation}
where $A_*(X)$ is the Chow theory of $X$ (with $\Z$ coefficients).
The Chow group  
$$A_n(\P^{d_1} \times \cdots \times \P^{d_r})$$
is generated by linear subvarieties
$$\iota_{m_1,\ldots,m_r}: \P^{m_1} \times \cdots \times \P^{m_r}
\hookrightarrow
\P^{d_1} \times \cdots \times \P^{d_r}$$
where $\sum_{i=1}^r m_i = n$.
We conclude $[f]$ is a $\Z$-linear combination of the elements
$$[\iota_{m_1,\ldots,m_r}] \in \widetilde{\omega}_{n}(
\P^{d_1} \times \cdots \times \P^{d_r}
).$$

Relations in $\omega_n(\P^{d_1} \times \cdots \times \P^{d_r})$
 lift canonically to $\omega_{n,1^r}(\P^{d_1} \times \cdots \times \P^{d_r})$
by pulling-back the list 
\begin{equation} \label{ktt4}
\sO_{\P^{d_1}}(1), \ldots, \sO_{\P^{d_r}}(1)
\end{equation}
everywhere. Since all double point relations in $\omega_n(\P^{d_1} \times \cdots \times \P^{d_r})$ occur over $\P^{d_1} \times \cdots \times \P^{d_r}$,
the pull-back of the list \eqref{ktt4} is well-defined and canonical.
The pull-back of the list \eqref{ktt4} via
$\iota_{m_1,\ldots, m_r}$ yields the element
$$\phi(\lambda,(m_1,\ldots,m_r)) \in \omega_{n,1^r}(k)\ ,$$
where $\sum_{i=1}^rm_i = n$.
Here, $\lambda$ is obtained simply by removing
the 0 parts $m_i$.
Hence,  after pushing-forward 
from $\P^{d_1} \times \cdots \times \P^{d_r}$
to $\Spec(k)$,
the argument is complete.
\end{proof}

\subsection{Projective bundles}
We will need
auxiliary results on projective bundles 
to remove the global generation hypothesis  of
Proposition \ref{popo}.

Let $Z\in \Sm_k$ be a projective variety equipped with a list of
line bundles $L_1,\ldots,L_r$ and a split rank 2
vector bundle 
$$B=\sO_Z \oplus N.$$
We are interested in the classes
$$
[\P(B), L_1,\ldots,L_r], \ 
[\P(B), L_1(\pm 1),\ldots,L_r(\pm1)] \ \in \omega_{*,1^r}(k)\ .
$$
Here, $\P(B)$ denotes the projectivization by sub-lines, and 
 $L_i(\pm1)$ stands for $L_i\otimes \sO_{\P(B)}(\pm 1)$.

Let $s$ be the section $Z \rightarrow \P(B)$ determined by the factor
$N \subset B$. 
The divisor $s$ is an element of the linear series associated
to $\sO_{\P(B)}(1)$.
The degeneration to the normal cone of $s$ 
yields a double point relation in 
$\omega_{*}(Z)$. After pulling-back the list $L_1, \ldots, L_r$,
we obtain a double point relation in $\omega_{*,1^r}(Z)$.
 Twisting the list by the exceptional divisor of
the degeneration  yields the relation
\begin{eqnarray*}
\ \ [\P(B), L_1,\ldots,L_r]& &\\
-
[\P(B), L_1(1),\ldots,L_r(1)]& & \\
-[\P(\sO_Z\oplus N^*), L_1(-1), \ldots L_r(-1)]& &\\
+[\P(B), L_1\otimes N^*,\ldots,L_r\otimes N^*]& = & 0 \ \in \omega_{*,1^r}(Z)\
.
\end{eqnarray*}

Since $(\sO_Z \oplus N^*) \otimes N \cong B$, we may rewrite the
above relation in the following form:
\begin{eqnarray*} 
\ \  [\P(B), L_1,\ldots,L_r]& &\\
-
[\P(B), L_1(1),\ldots,L_r(1)]& & \\
-[\P(B), L_1\otimes N^*(-1), \ldots L_r\otimes N^*(-1)]& &\\
+[\P(B), L_1\otimes N^*,\ldots,L_r\otimes N^*]& = & 0
\ \in \omega_{*,1^r}(Z)\
.
\end{eqnarray*}
After replacing $L_i$ with $L_i \otimes N$ everywhere,
we obtain our main projective bundle
relation in $\omega_{*,1^r}(Z)$:
\begin{eqnarray*}\nonumber
[\P(B), L_1(-1), \ldots L_r(-1)]& =&
  \ \ [\P(B), L_1\otimes N,\ldots,L_r\otimes N]\\
& & - 
[\P(B), L_1\otimes N(1),\ldots,L_r\otimes N(1)]\\ & &
+[\P(B), L_1, \ldots,L_r]\  \nonumber 
 .
\end{eqnarray*}

\begin{prop} Let $Y\in \Sm_k$ be a projective variety of dimension
$n$ with \label{popop}
arbitrary line bundles $L_1, \ldots, L_r$ .
Then,
$$[Y, L_1, \ldots, L_r] \in \widetilde{\omega}_{n,1^r}(k) $$
lies in the $\Z$-linear span of 
$$\left\{\ \widetilde{\phi}(\lambda,(m_1,\ldots,m_r)) \ \Big| 
 \ (\lambda,(m_1,\ldots,m_r))\in \sP_{n,1^r}, \
\sum_{i=1}^r m_i =n \ \right\}$$
in $\widetilde{\omega}_{n,1^r}(k)$.
\end{prop}

\begin{proof}
Let $Z\subset Y$ be a nonsingular divisor such that
$L_1(Z), \ldots, L_r(Z)$ are all globally generated.
Consider the double point relation in $\omega_{n,1^r}(Y)$
obtained from degenerating to the normal cone of
$Z$, pulling-back the list $L_1,\ldots, L_r$,
and twisting by the exceptional divisor of the degeneration:
\begin{eqnarray} \label{fred}
\ \ [Y, L_1,\ldots,L_r]& &\\ \nonumber
-
[Y, L_1(Z),\ldots,L_r(Z)]& & \\ \nonumber
-[\P(\sO_Z\oplus \sO_Z(Z)), L_1(-1), \ldots L_r(-1)]& &\\ \nonumber
+[\P(\sO_Z\oplus \sO_Z(Z)), L_1(Z),\ldots,L_r(Z)]& = & 0 \ 
\in \omega_{n,1^r}(Y)\
.
\end{eqnarray}

Proposition \ref{popo} applies to the second and fourth term of
relation \eqref{fred}. The third term, however, requires
further analysis.
Using our main projective bundle relation in $\omega_{n,1^r}(Z)$,
we can trade the third term for
\begin{eqnarray*}\nonumber
 -[\P(\sO_Z \oplus \sO_Z(Z)), L_1(Z),\ldots,L_r(Z)]&\\
+
[\P(\sO_Z \oplus \sO_Z(Z) ), L_1(Z)(1),\ldots,L_r(Z)(1)]&\\
-[\P(\sO_Z \oplus \sO_Z(Z) ), L_1, \ldots,L_r]&.  \nonumber 
\end{eqnarray*}
The last two terms are not covered by Proposition \ref{popo}.

We have proven the Proposition modulo elements of the form
$$[\P(B), L'_1, \ldots, L'_r], \ \ [\P(B),L'_1(1), \ldots, L'_r(1)] 
\in \omega_{n,1^r}(Z) $$
where $B= \sO_Z \oplus N$ is a split rank 2 bundle and
$L'_i$ are arbitrary line bundles on $Z$.
Let 
$$\pi: \P(B) \rightarrow Z$$
be the projection.
Let $Z'\subset Z$ be a nonsingular divisor such that
$$L'_1(Z'), \ldots, L'_r(Z'),\ L'_1(Z')(1), \ldots, L'_r(Z')(1)
$$ are all globally generated on $\P(B)$.

Consider the double point relation in $\omega_{n,1^r}(Z)$
obtained from degenerating to the normal cone of
$\pi^{-1}(Z')\subset \P(B)$, pulling-back the list $L'_1,\ldots, L'_r$,
and twisting by the exceptional divisor of the degeneration:
\begin{eqnarray} \label{ffred}
\ \ [\P(B), L'_1,\ldots,L'_r]& &\\ \nonumber
-
[\P(B), L'_1(Z'),\ldots,L'_r(Z')]& & \\ \nonumber
-[\P(B_{Z'}) \times_{Z'}
\P(\sO_{Z'}\oplus \sO_{Z'}(Z')), L'_1(0,-1), \ldots L'_r(0,-1)]& &\\ \nonumber
+[\P(B_{Z'}) \times_{Z'}\P(\sO_{Z'}\oplus \sO_{Z'}(Z')), L'_1(Z'),\ldots,L'_r(Z')]& = & 0 
\end{eqnarray}
in $\omega_{n,1^r}(Z)$
. A similar relation holds for $[\P(B), L'_1(1),\ldots,L'_r(1)]$.
We treat the third term of \eqref{ffred} in both cases with our
main projective bundle relation for the $\P(\sO_{Z'} \oplus
\sO_{Z'}(Z'))$ projectivization.

We now have proven the Proposition modulo elements of the form
\begin{eqnarray*}
{}[\P(B_1)\times_{Z'} \P(B_2), L''_1, \ldots, L''_r], & \\
 {}[\P(B_1)\times_{Z'} \P(B_2),L''_1(1,0), \ldots, L''_r(1,0)], & \\
{}[\P(B_1)\times_{Z'} \P(B_2), L''_1(0,1), \ldots, L''_r(0,1)], & \\ 
{}[\P(B_1)\times_{Z'} \P(B_2)      ,L''_1(1,1), \ldots, L''_r(1,1)]\ &\in
\omega_{n,1^r}(Z')
\end{eqnarray*}
where $B_i= \sO_{Z'} \oplus N_i$ are split rank 2 bundles and
$L''_i$ are arbitrary lines bundles on $Z'$.

We iterate the procedure by selecting a sufficiently
positive divisor $Z''\subset Z'$. Since the dimensions 
of the divisors are dropping, the procedure terminates 
when dimension 0 is reached with
the elements
$$[\underbrace{\P^1\times\cdots \times \P^1}_{n}, 
\underbrace{\sO(l_1,\ldots,l_n),
\ldots, \sO(l_1,\ldots,l_n)}_{r}] \in \omega_{n,1^r}(k)$$
with $l_i\in\{0,1\}$. These elements  are covered by Proposition \ref{popo}.
\end{proof}

\subsection{Proof of Theorem \ref{kk34}} \label{gaga}
We prove the result by induction on $n$. The $n=0$
case is clear. We assume the result for all $n'<n$.

Using Theorem \ref{kk34} for $n'<n$, we conclude the 
grade $n$ part of 
$$\mathfrak{m} \cdot \omega_{*,1^r} \otimes_\Z \Q$$
is equal to the 
$\Q$-linear span of 
$$\left\{\ \widetilde{\phi}(\lambda,(m_1,\ldots,m_r)) \ \Big| 
\
 (\lambda,(m_1,\ldots,m_r))\in \sP_{n,1^r},\ 
\sum_{i=1}^r m_i <n \ \right\}$$
in ${\omega}_{n,1^r}(k)\otimes_\Z \Q$.
By Proposition \ref{popop}, we see
$$\text{dim}(\omega_{n,1^r}(k)\otimes_\Z \Q) \leq |\sP_{n,1^r}| \ .$$
Since we have already established the reverse inequality in  Section
\ref{ci},
we obtain 
$$\text{dim}(\omega_{n,1^r}(k)\otimes_\Z \Q) = |\sP_{n,1^r}| \ ,$$
concluding the proof of Theorem \ref{kk34}.
\qed

\section{Higher rank} \label{hr}
\subsection{Splitting}
As before, let $\mathfrak{m} \subset \omega_*(k)$ be the ideal generated
by all elements of positive dimension.
 Since $\omega_{*,r}(k)$ is a $\omega_*(k)$-module, we can define the
graded quotient 
$$\widetilde{\omega}_{*,r}(k) = \frac{\omega_{*,r}(k)}
{\mathfrak{m}\cdot \omega_{*,r}(k)}
\ , 
\ \ \ \ \ \ 
\widetilde{\omega}_{*,r}(k) = 
\bigoplus_{n=0}^\infty  \widetilde{\omega}_{n,r}(k)\ .$$
For $(\lambda,\mu)\in \sP_{n,r}$, let
$$\ \widetilde{\phi}(\lambda,\mu) \in \widetilde{\omega}_{n,r}(k)$$
denote the class of ${\phi}(\lambda,\mu)$ in the quotient.

\begin{prop} Let $Y\in \Sm_k$ be a projective variety of dimension
$n$ with \label{popopp}
rank $r$ vector bundle $E$.
Then,
$$[Y, E] \in \widetilde{\omega}_{n,r}(k) $$
lies in the $\Z$-linear span of 
$$\left\{\ \widetilde{\phi}(\lambda,\mu) \ \Big| 
 \ (\lambda,\mu)\in \sP_{n,r}, \
|\mu| =n \ \right\}$$
in $\widetilde{\omega}_{n,r}(k)$.
\end{prop}

For the proof of Proposition \ref{popopp}, we will require
the following basic result.

\begin{lem} \label{p:2}
There exists a nonsingular projective variety $\widehat{Y}$ and a
birational morphism
$$\widehat{Y} \rightarrow Y$$
for which 
the pull-back of $E$ to $\widehat{Y}$
has a filtration by sub-bundles
\[
 0=E_0 \subset E_1 \subset E_2 \subset \ldots \subset
 E_r = E
\]
satisfying $\text{\em rank}(E_i/E_{i-1})=1$.
\end{lem}

\begin{proof}
Consider the complete flag variety over $Y$,
$$\pi: {\mathbb{F}}(E) \to Y \ .$$
There is a rational section $s$ of $\pi$.
The variety $\widehat{Y}$ is obtained from the
resolution of singularities of the graph closure of $s$ in 
$Y \times \mathbb{F}(E)$.
\end{proof}

To prove Proposition \ref{popopp}, let $[Y,E]$ be given.
Since 
$$[\widehat{Y}\rightarrow Y ] 
= [Y\rightarrow Y] \in \widetilde{\omega}_{n}(Y)$$
by \eqref{gt5}, 
we conclude
$$[\widehat{Y}\rightarrow Y, E ] = [Y\rightarrow Y,E] \in 
\widetilde{\omega}_{n,r}(Y)$$
as before. After pushing-forward to $\Spec(k)$, we obtain
$$[\widehat{Y}, E ] = [Y,E] \in 
\widetilde{\omega}_{n,r}(k).$$

On $\widehat{Y}$, let $L_1, \ldots, L_r$ be the list of
line bundle obtained from the subquotients of the
filtration of $E$.
Sending the extension parameters to 0, we see
$$[\widehat{Y}, E ] = [\widehat{Y}, L_1\oplus \cdots \oplus L_r] \in 
{\omega}_{n,r}(k).$$
Finally, Proposition \ref{popop} applied to the list
$[\widehat{Y}, L_1, \ldots, L_r]$ concludes the
proof of Proposition \ref{popopp}. \qed

\subsection{Proofs of Theorems \ref{one} and \ref{tttt}}

We prove the result by induction on $n$. The $n=0$
case is clear. We assume the result for all $n'<n$.

Using Theorem \ref{one} for $n'<n$, we conclude the 
grade $n$ part of 
$$\mathfrak{m} \cdot \omega_{*,r} \otimes_\Z \Q$$
is equal to the 
$\Q$-linear span of 
$$\left\{\ \widetilde{\phi}(\lambda, \mu) \ \Big| 
\
 (\lambda, \mu)\in \sP_{n,r},\ 
|\mu| <n \ \right\}$$
in ${\omega}_{n,r}(k) \otimes_\Z \Q$.
By Proposition \ref{popopp}, we see
$$\text{dim}(\omega_{n,r}(k)\otimes_\Z \Q) \leq |\sP_{n,r}| \ .$$
Since we have already established the reverse inequality in Section
\ref{bilp},
we obtain 
$$\text{dim}(\omega_{n,r}(k)\otimes_\Z \Q) = |\sP_{n,r}| \ ,$$
concluding the proof of Theorems \ref{one} and \ref{tttt}.
\qed

\subsection{Proof of Theorem \ref{lltt}} \label{lltt2}
Since Proposition \ref{popopp} holds over $\Z$,
we see $\omega_{n,r}(k)$ is generated over
$\Z$ by 
$$\left\{\ {\phi}(\lambda,\mu) \ \Big| 
 \ (\lambda,\mu)\in \sP_{n,r}, \
|\mu| =n \ \right\}$$
and the subgroups
$$\omega_{n}(k)\cdot \omega_{0,r}(k), \ \ 
\omega_{n-1}(k)\cdot \omega_{1,r}(k),\ \ \ldots,
\ \ \omega_{1}(k) \cdot \omega_{n-1,r}(k)\ .$$

We now prove Theorem \ref{lltt} by induction on $n$.
Certainly, $\omega_{i}(k)$ is a free $\Z$-module of rank equal to
the number of partitions of $i$. Using the induction
hypothesis, we see 
$\omega_{n,r}(k)$ has $|\sP_{n,r}|$ generators over $\Z$.
Since we know 
$$\text{dim}(\omega_{n,r}(k)\otimes_{\Z}\Q) = |\sP_{n,r}|,$$
no relations among these generators are possible.
\qed

\subsection{Product structure}
There is a natural commutative ring structure on 
$$\omega_{*,+}(k) = \bigoplus_{r=1}^\infty \omega_{*,r}(k) = 
\bigoplus_{n=0}^\infty \bigoplus_{r=1}^\infty \omega_{n,r}(k)$$
given by external product\ 
$$[Y_1, E_1] \cdot [Y_2,E_2] = [ Y_1\times Y_2,\ p_1^*(E_1) \otimes
p_2^*(E_2)]\ .$$
Here, $p_1$ and $p_2$ are the projections of $Y_1\times Y_2$
onto the first and second factors respectively.
There is an inclusion of rings
$$\omega_{*}(k) \hookrightarrow \omega_{*,+}(k), \ \ \ \ [Y] \mapsto [Y,\sO] 
\ .$$

By the basis result of Theorem \ref{one},
the product on 
$\omega_{*,+}(k)\otimes_\Z \Q$ is completely determined by 
the special case
\begin{equation*}
[\P^a, \sO(1)] \cdot [\P^b, \sO(1)] = 
[\P^a\times \P^b, \sO(1,1)]
\ .
\end{equation*}

\vspace{8pt}
\noindent{\bf Question.} {\em What is the decomposition of 
$[\P^a\times \P^b, \sO(1,1)]$ in the basis of $\omega_{a+b,1}(k)\otimes_\Z \Q$
given in Theorem \ref{one} ?}

\vspace{8pt}
Of course, Theorem \ref{tttt} provides a computational 
approach to the question for any fixed $a$ and $b$.
Is there a closed formula or any structure in the answer?

\section{Results over $X$}
\subsection{Surjectivity}
Following the notation of Section \ref{glob}, let
$$\widetilde{\omega}_{*,1^r}(X)
 = \frac{\omega_{*,1^r}(X)}{\mathfrak{m}\cdot \omega_{*,1^r}(X)}
\ , 
\ \ \ \ \ \ 
\widetilde{\omega}_{*,1^r}(X)
 = \bigoplus_{n=0}^\infty  \widetilde{\omega}_{n,1^r}(X)\ .$$
Consider the element
 $$[Y\rightarrow X, L_1,\ldots, L_r]\in \omega_{n,1^r}(X)\ .$$
If all the $L_i$ are globally generated on $Y$, then
there exists a projective morphism
$$f: Y \rightarrow X \times \P^{d_1} \times \cdots \times
\P^{d_r}, \ \ \ L_i=f^*(\sO_{\P^{d_i}}(1))\ .$$
We view $f$ as determining an element of algebraic cobordism,
$$[f: Y \rightarrow X\times \P^{d_1} \times \cdots \times \P^{d_r}] \in
\omega_n(X\times \P^{d_1} \times \cdots \times \P^{d_r})\ . $$

The Chow group  
$A_n(X\times \P^{d_1} \times \cdots \times \P^{d_r})$
is generated over $A_*(X)$ by linear subvarieties
$$\iota_{m_1,\ldots,m_r}: \P^{m_1} \times \cdots \times \P^{m_r}
\hookrightarrow
\P^{d_1} \times \cdots \times \P^{d_r}$$
where $\sum_{i=1}^r m_i \leq n$. Using \eqref{gt5} for
$X\times \P^{d_1} \times \cdots \times \P^{d_r}$,
we see $[f]$ is a $\Z$-linear combination of elements of the form
$$[\iota\times \iota_{m_1,\ldots,m_r}] \in \widetilde{\omega}_{n}(X\times
\P^{d_1} \times \cdots \times \P^{d_r}
)$$
where 
$\iota: W\rightarrow X$
is a resolution of singularities of an irreducible
subvariety of $X$ and
$$n= \text{dim}(W) + \sum_{i=1}^r m_i\ .$$
Concluding as in the proof of Proposition \ref{popo}, we find
$$[Y\rightarrow X, L_1,\ldots, L_r]\in \widetilde{\omega}_{n,1^r}(X)$$
lies in the 
subspace spanned by
products of elements of $\omega_\delta(X)$ with basis terms of
$\omega_{n-\delta,1^r}(k)$.

The projective bundle analysis in the
proof of Proposition \ref{popop} occurs entirely over $Y$ and
thus over $X$. Hence, we
can remove the global generation hypothesis on the bundles $L_i$ just as before.

Since the splitting of Lemma  \ref{p:2} also occurs over $Y$,
we conclude the composition
$$\omega_{*}(X) \otimes_{\omega_*(k)}\omega_{*,r}(k) 
\stackrel{\gamma_X}{\longrightarrow} \omega_{*,r}(X)\longrightarrow
\widetilde{\omega}_{*,r}(X)$$
is surjective. 

\begin{prop} The natural map \label{ht1}
$$\gamma_X: \omega_{*}(X) \otimes_{\omega_*(k)}\omega_{*,r}(k) 
\rightarrow \omega_{*,r}(X)$$
is surjective.
\end{prop}

\begin{proof}
We have already seen $\gamma_X$ surjects onto $\omega_{*,r}(X)/
\mathfrak{m}\cdot \omega_{*,r}(X)$. But then,
$$\mathfrak{m} \cdot \omega_{*}(X) \otimes_{\omega_*(k)}\omega_{*,r}(k)$$
surjects via $\gamma_X$ onto 
 $$\frac{\mathfrak{m} \cdot\omega_{*,r}(X)}
{\mathfrak{m}^2\cdot \omega_{*,r}(X)}\ .$$
The result follows by iteration since $\bigcap_{i\geq 1} \mathfrak{m}^i = 0$.
\end{proof}

\subsection{Injectivity}\label{inj}
Let $c_1,\ldots, c_r$ be variables with $c_i$ of degree $i$. Let
 $\Psi$ be the space of polynomials 
in $c_1,\ldots, c_r$ with $\Z$-coefficients.
For homogeneous $\psi\in \Psi$ of degree $d$,
there are natural Chern operations
$$C_{\psi}: \omega_{*,r}(X) \rightarrow \omega_{*-d}(X)$$
defined by
\begin{multline}\label{llqq}
C_{\psi}( [Y \stackrel{f}{\longrightarrow} X, E])=  
\\ f_*\Big(\psi(c_1(E), \ldots, c_r(E)) \cap
[Y\rightarrow Y]\Big) \in \omega_{*-d}(X) 
\end{multline}
where the action of $\psi$
on the right is via the standard Chern class operations 
\cite[Section 7.4]{ACFin}
in algebraic cobordism.

To show definition \eqref{llqq} respects the double point relation
in $\omega_{*,r}(X)$, we argue as follows.
Suppose 
\[
g: Y \to X \times \P^1
\]
is a projective morphism
for which the 
composition
\begin{equation*}
\pi = p_2\circ g: Y \to \P^1
\end{equation*}
is 
a double point degeneration over $0 \in \P^1$, and
$E$ is a rank $r$ vector bundle on $Y$.
The Chern operation $\psi(c_1(E), \ldots, c_r(E))$
is well-defined on $\omega_{*}(Y)$,
$$\psi: \omega_*(Y) \rightarrow \omega_{*-d}(Y)\ .$$
Hence, for regular values $\zeta\in \P^1(k)$ of $\pi$, 
$$\psi \cap \Big( [Y_\zeta\rightarrow Y] -
[A \rightarrow Y] - [B \rightarrow Y] + [\P(\pi)\rightarrow Y]\Big)=
0 \in \omega_*(Y)\ .$$ Pushing-forward to $X$ and using the 
functoriality of the Chern class, we obtain
\begin{multline*}
C_\psi([Y_\zeta \to X, E_{Y_\zeta}]) -
C_\psi([A \to X, E_A]) \\
- C_\psi([B \to X, E_B]) + 
C_\psi([\P(\pi) \to X, E_{\P(\pi)}]) = 0 \in \omega_{*}(X) \  
\end{multline*}
which is the required compatibility.

By the characterization of $\omega_{*,r}(k)$ in Theorem \ref{lltt}, we
have
\begin{equation}\label{jtw}
\omega_{*}(X) \otimes_{\omega_*(k)}
\omega_{*,r}(k) = \bigoplus_{\lambda}\  \omega_{*}(X) \otimes \phi(\lambda,\lambda)
\end{equation}
where the sum is over all partitions $\lambda$ of length at
most $r$.
Consider the pairing
$$\rho^X: \omega_{*}(X) \otimes_{\omega_*(k)}
\omega_{*,r}(k) \times \Psi \rightarrow \omega_*(X)$$
defined by
$$\rho^X\big( \zeta,\psi ) = C_\psi(\gamma_X(\zeta))\ .$$
Using the basis \eqref{jtw}, we see
the pairing $\rho^X$ is triangular with 1's on the diagonal
by calculation \eqref{ppp3}. We have proven the
following result.

\begin{prop} \label{ht2} The natural map
$$\gamma_X: \omega_{*}(X) \otimes_{\omega_*(k)}\omega_{*,r}(k) 
\rightarrow \omega_{*,r}(X)$$
is injective.
\end{prop}

Propositions \ref{ht1} and \ref{ht2} together complete the proof of
Theorem \ref{uuu}. In fact, the proof of Theorem \ref{uuu} is just a slight
abstraction of the original proof of Theorem \ref{one}.

\vspace{+12 pt}
\noindent
Department of Mathematics \\
University of Utah\\ 
yplee@math.utah.edu \\
\vspace{+8 pt}

\noindent
Department of Mathematics\\
Princeton University\\
rahulp@math.princeton.edu

\end{document}